\documentclass[12pt]{article}    
\usepackage[margin=0.75in]{geometry} 
\usepackage{amsmath,amssymb,amsthm}
\newtheorem{theorem}{Theorem}[section]

\newtheorem{lemma}[theorem]{Lemma}

\newtheorem{remark}{Remark}

\newcommand{\al}{\alpha}
\newcommand{\bt}{\beta}

\newcommand{\s}{\sigma}

\newcommand{\be}{\begin{equation}}
\newcommand{\ee}{\end{equation}}
\newcommand{\bea}{\begin{eqnarray}}
\newcommand{\eea}{\end{eqnarray}}
\newcommand{\no}{\nonumber}

\numberwithin{equation}{section}
\begin{document}
\title{Orthogonal polynomials for the singularly perturbed Laguerre weight, Hankel determinants and asymptotics}
\author{Chao Min\thanks{School of Mathematical Sciences, Huaqiao University, Quanzhou 362021, China; Email: chaomin@hqu.edu.cn}\: and Xiaoqing Wu\thanks{School of Mathematical Sciences, Huaqiao University, Quanzhou 362021, China}}


\date{August 7, 2026}
\maketitle
\begin{abstract}
Based on the work of Chen and Its [{\em J. Approx. Theory} {\bf 162} ({2010}) {270--297}], we further study orthogonal polynomials with respect to the singularly perturbed Laguerre weight $w(x;t,\alpha) = {x^\alpha}{\mathrm e^{ - x-\frac{t}{x}}}, \; x\in\mathbb{R}^{+},\;\alpha > -1,\; t\geq  0$. By using the ladder operators and associated compatibility conditions for orthogonal polynomials with general Laguerre-type weights, we derive the second-order differential equation satisfied by the orthogonal polynomials, a system of difference equations and a system of differential-difference equations for the recurrence coefficients. We also investigate the properties of the zeros of the orthogonal polynomials. Using Dyson's Coulomb fluid approach together with the discrete system, we obtain the large $n$ asymptotic expansions of the recurrence coefficients, the sub-leading coefficient of the monic orthogonal polynomials, the Hankel determinant and the normalized constant for fixed $t>  0$. It is found that all the asymptotic expansions are singular at $t=0$. We also study the long-time ($t\rightarrow+\infty$) asymptotics of these quantities explicitly for fixed $n\in\mathbb{N}$ from the Toda-type system.
\end{abstract}

$\mathbf{Keywords}$: Singularly perturbed Laguerre weight; Orthogonal polynomials; Zeros;

Hankel determinants; Large $n$ asymptotics; Long-time asymptotics.

$\mathbf{Mathematics\:\: Subject\:\: Classification\:\: 2020}$: 42C05, 33C45, 41A60.

\section{Introduction}
In the significant work \cite{CI2010}, Chen and Its studied a singular linear statistic of the Laguerre unitary ensemble (LUE) inspired by an integrable quantum field theory at finite temperature \cite{Lu}. This problem is equivalent to the study of the Hankel determinant and orthogonal polynomials for a singularly perturbed Laguerre weight
\begin{equation}\label{w1}
w(x;t,\al) := {x^\alpha }{\mathrm e^{ - x-\frac{t}{x}}}, \qquad x\in\mathbb{R}^{+},
\end{equation}
with parameters $\alpha > -1,\; t \ge  0$. (Actually, Chen and Its \cite{CI2010} only considered the $\al>0$ case.) The Hankel determinant can be interpreted as the moment generating function of the Wigner delay time \cite{MS,TM}.

The matrix model and Hankel determinant associated with the weight (\ref{w1}) have played an important role in the study of statistics for zeros of the Riemann zeta function \cite{BS} and the bosonic replica field theories \cite{OK}. Orthogonal polynomials and Hankel determinants for singularly perturbed Gaussian and Jacobi weights have also been extensively studied in the past few years; see, e.g., \cite{BMM,ChenDai,MM,MC2020,MLC}. Matrix models with higher-order poles have been investigated in \cite{Atkin,DXZ2018,DXZ2019}.

We would like to mention that orthogonal polynomials and Hankel determinants for the deformed measure $\mathrm e^{ \frac{t}{x}}d\mu(x)$ have a very close relationship with the peakon solutions of the Camassa-Holm equation \cite{Beals1999,Beals2000,Beals2001}. For more recent work in this respect, see \cite{CC2024,CCH2014,CHS2016}. Orthogonal polynomials and Hankel determinants for the singularly perturbed Laguerre weight (\ref{w1}) have been used to study the long-time asymptotics of the peakon solutions of the Camassa-Holm equation very recently by Chang, Eckhardt and Kostenko \cite[Section 6]{Chang}.

The main results of \cite{CI2010} are that the diagonal recurrence coefficient of the monic orthogonal polynomials satisfies a particular third Painlev\'{e} equation and the logarithmic derivative of the Hankel determinant satisfies the Jimbo-Miwa-Okamoto $\s$-form of the third Painlev\'{e} equation under suitable transformations. The results are derived by using the ladder operator approach and the Riemann-Hilbert approach independently.

Based on the Deift-Zhou steepest descent method for Riemann-Hilbert problems, Xu, Dai and Zhao \cite{Xu1} studied the scaled eigenvalue correlation kernel for the singularly perturbed LUE associated with the weight (\ref{w1}) and found that the limit is a new kernel called $\Psi$-kernel, which has the transition to the Bessel kernel and Airy kernel for small and large parameters, respectively. Subsequently, Xu, Dai and Zhao \cite{Xu2} studied the large $n$ asymptotics of the Hankel determinant together with the leading coefficient of the corresponding orthonormal polynomials and the recurrence coefficients for the monic orthogonal polynomials.

Classical orthogonal polynomials (such as Hermite, Laguerre and Jacobi polynomials) are orthogonal with respect to a weight $w(x)$ on the real line which satisfies the Pearson equation
\begin{equation}\label{r}
 \frac{d}{{dx}}(\sigma (x)w(x)) = \tau (x)w(x),
\end{equation}
where $\sigma (x)$ and $\tau (x)$ are polynomials with deg $\sigma \le 2$ and deg $\tau$=1. For semi-classical orthogonal polynomials, they have a weight $w(x)$ that satisfies the Pearson equation (\ref{r}), where $\sigma (x)$ and $\tau (x)$ are polynomials with deg $\sigma > 2$ or deg $\tau\neq 1$. See, e.g., \cite[Section 1.1.1]{W}.
It is not difficult  to find that (\ref{w1}) is a semi-classical weight, since it satisfies the Pearson equation (\ref{r}) with
$$
\sigma (x) = {x^2}, \qquad \tau (x) =  - {x^2} + (2 +  \alpha)x + t.
$$

Let $P_n(x; t,\al), n=0, 1, 2,\ldots$ be the monic polynomials of degree $n$ orthogonal with respect to the weight (\ref{w1}) and $P_n(x; t,\al)$ has the expansion
\be\label{ex}
P_n(x; t,\al)=x^n+\mathrm{p}(n, t,\al) x^{n-1}+\cdots+P_n(0; t,\al),
\ee
where $\mathrm{p}(n, t,\al)$ is the sub-leading coefficient of $P_n(x; t,\al)$
with the initial value $\mathrm{p}(0, t,\al)=0$. The orthogonality condition reads
\begin{equation}\label{h_j}
\int_{0}^{\infty}P_j(x;t,\al)P_k(x;t,\al)w(x;t,\al)dx= h_j(t,\al)\delta_{jk},\qquad j, k=0,1,2,\ldots,
\end{equation}
where $h_j(t,\al)>0$ and $\delta_{jk}$ is the Kronecker delta. In this paper, we call $h_j(t,\al)$ the normalized constant following Fokas, Its and Kitaev's terminology \cite[p. 396]{Fokas2}.

The orthogonal polynomials satisfy the three-term recurrence relation of the form
\begin{equation}\label{xp}
xP_{n}(x;t,\al)=P_{n+1}(x;t,\al)+\al_n(t,\al) P_n(x;t,\al)+\beta_{n}(t,\al)P_{n-1}(x;t,\al),
\end{equation}
with initial conditions $P_0(x;t,\al)=1,\; \beta_0(t,\al) P_{-1}(x;t,\al)=0$. It follows that the recurrence coefficients $\al_n(t,\al)$ and $\beta_{n}(t,\al)$ have the integral representations
\begin{align}
&\al_n(t,\al)=\frac{1}{h_n(t,\al)}\int_{0}^{\infty}xP_{n}^2(x;t,\al)w(x;t,\al)dx,\no\\[5pt]
&\bt_n(t,\al)=\frac{1}{h_{n-1}(t,\al)}\int_{0}^{\infty}xP_{n}(x;t,\al)P_{n-1}(x;t,\al)w(x;t,\al)dx.\no
\end{align}
Please note that we use the notation $\al_n, \bt_n$ for the recurrence coefficients (since this is standard), which should not cause confusion with the Laguerre weight parameter $\al$ as the latter is never subscripted with $n$.

From (\ref{ex}), (\ref{h_j}) and (\ref{xp}), the recurrence coefficients also have the following expressions:
\begin{align}\label{alpha1}
&\alpha_{n}(t,\al)=\mathrm{p}(n,t,\al)-\mathrm{p}(n+1,t,\al),\\[5pt]
&\beta_{n}(t,\al)=\frac{h_{n}(t,\al)}{h_{n-1}(t,\al)}.\label{be}
\end{align}
In view of (\ref{alpha1}), we get an important identity
\be\label{pn}
\sum_{j=0}^{n-1}\al_j(t,\al)=-\mathrm{p}(n,t,\al).
\ee
Moreover, using the three-term recurrence relation (\ref{xp}) and with the aid of (\ref{be}), we have the famous Christoffel-Darboux formula
$$
\sum_{j=0}^{n-1}\frac{P_j(x;t,\al)P_j(y;t,\al)}{h_j(t,\al)}=\frac{P_n(x;t,\al)P_{n-1}(y;t,\al)-P_n(y;t,\al)P_{n-1}(x;t,\al)}{h_{n-1}(t,\al)(x-y)},\qquad x\neq y.
$$
When $x=y$, the above formula becomes
$$
\sum_{j=0}^{n-1}\frac{P_j^2(x;t,\al)}{h_j(t,\al)}=\frac{P_n'(x;t,\al)P_{n-1}(x;t,\al)-P_n(x;t,\al)P_{n-1}'(x;t,\al)}{h_{n-1}(t,\al)},
$$
where $'$ denotes $\frac{d}{dx}$.

It is also known that the orthogonal polynomials can be expressed as the determinant
\be\label{pnd}
P_n(x;t,\al)=\frac{1}{D_n(t,\al)}\begin{vmatrix}
\mu_{0}(t,\al)&\mu_{1}(t,\al)&\cdots&\mu_{n}(t,\al)\\
\mu_{1}(t,\al)&\mu_{2}(t,\al)&\cdots&\mu_{n+1}(t,\al)\\
\vdots&\vdots&&\vdots\\
\mu_{n-1}(t,\al)&\mu_{n}(t,\al)&\cdots&\mu_{2n-1}(t,\al)\\
1&x&\cdots&x^n
\end{vmatrix}
\ee
and the multiple integral
\be\label{pnmi}
P_n(x;t,\al)=\frac{1}{n! D_n(t,\al)}\int_{(0,\infty)^{n}}\prod_{1\leq i<j\leq n}(x_j-x_i)^2\prod_{k=1}^n (x-x_k)w(x_k;t,\al) dx_k,
\ee
where $D_n(t,\al)$ is the Hankel determinant for the weight (\ref{w1}) defined by
$$
{D_n}(t,\al): = \det (\mu _{i + j}(t,\al))_{i,j = 0}^{n - 1} = \left| \begin{array}{cccc}
\mu _0(t,\al)&\mu _1(t,\al)&\cdots&\mu _{n - 1}(t,\al)\\
\mu _1(t,\al)&\mu _2(t,\al)&\cdots&\mu _n(t,\al)\\
 \vdots & \vdots &{}& \vdots \\
\mu _{n - 1}(t,\al)&\mu _n(t,\al)&\cdots&\mu _{2n - 2}(t,\al)
\end{array} \right|
$$
and $\mu_j(t,\al)$ is the $j$th moment given by
$$
\mu_j(t,\al):=\int_{0}^{\infty}x^j w(x;t,\al)dx,\qquad j=0,1,2,\ldots.
$$
We set $D_0(t,\al)=1$. An evaluation of the above integral gives
$$
\mu_j(t,\al)=2 t^{\frac{1}{2} (j+\alpha+1)} K_{j+\alpha +1}\big(2 \sqrt{t}\big),
$$
where $K_{\nu}(z)$ is the modified Bessel function of the second kind of order $\nu$ and it has the integral representation \cite[Equation (10.32.10)]{Olver}
$$
K_{\nu}(z)=\frac{1}{2}\left(\frac{z}{2}\right)^{\nu}\int_{0}^{\infty}\exp\left(-x-\frac{z^2}{4x}\right)\frac{dx}{x^{\nu+1}}.
$$

In view of (\ref{alpha1}) and (\ref{be}), and with the aid of (\ref{pnd}), the recurrence coefficients $\al_n(t,\al)$ and $\bt_n(t,\al)$ can be expressed as determinants of the forms
\be\label{ad}
\al_n(t,\al)=\frac{{\widetilde{D}_{n+1}}(t,\al)}{{D_{n+1}}(t,\al)}-\frac{{\widetilde{D}_n}(t,\al)}{{D_n}(t,\al)}
\ee
and
\begin{equation}\label{bd}
\bt_n(t,\al)=\frac{D_{n+1}(t,\al) D_{n-1}(t,\al)}{D_n^2(t,\al)},
\end{equation}
where ${\widetilde{D}_n}(t,\al)$ is obtained by replacing the last column $\left(\mu_{n - 1}(t,\al), \mu_{n}(t,\al), \ldots, \mu_{2n-2}(t,\al)\right)^{T}$ of ${D_n}(t,\al)$ by $\left(\mu_{n}(t,\al), \mu_{n+1}(t,\al), \ldots, \mu_{2n-1}(t,\al)\right)^{T}$. See, e.g., \cite[p. 3]{W}.

It is a well-known fact that the Hankel determinant $D_n(t,\al)$ has two alternative representations: the first as a product of the constants $h_j$'s for the monic orthogonal polynomials
\begin{equation}\label{d_nt1}
\begin{aligned}
{D_n}(t,\al)=&\prod\limits_{j = 0}^{n - 1} {{h_j}(t,\al)}
\end{aligned}
\end{equation}
and the second as a multiple integral
\be\label{mi}
D_n(t,\al)=\frac{1}{n!} \int_{(0, \infty)^n} \prod_{1 \leq i<j \leq n}(x_j-x_i)^2 \prod_{k=1}^n w(x_k;t,\al) d x_k.
\ee
The combination of (\ref{pnmi}) and (\ref{mi}) shows that
\be\label{pn0}
(-1)^nP_n(0;t,\al)=\frac{D_n(t,\al+1)}{D_n(t,\al)}.
\ee
For more information about orthogonal polynomials, see, for example, \cite{Chihara,Ismail,G}.

The rest of the paper is organized as follows. In Section 2, by applying the ladder operators and compatibility conditions for the generalized Laguerre-type weight to our problem, we show that the orthogonal polynomials satisfy a second-order differential equation, and derive a discrete system and a Toda-type system for the recurrence coefficients. We also present the relations among the logarithmic derivative of the Hankel determinant, the sub-leading coefficient of the monic orthogonal polynomials, the recurrence coefficients and the auxiliary quantities. In Section 3, we obtain a mixed recurrence relation for the orthogonal polynomials and then study some properties of the zeros of the orthogonal polynomials. In Section 4, we investigate the large $n$ asymptotics of the recurrence coefficients $\al_n(t,\al)$ and $\bt_n(t,\al)$, the sub-leading coefficient $\mathrm{p}(n, t,\al)$, the Hankel determinant ${D_n}(t,\al)$ and the normalized constant ${{h_n}(t,\al)}$ for fixed $t > 0$. We discuss the long-time ($t\rightarrow+\infty$) asymptotics of these quantities explicitly for fixed $n\in\mathbb{N}$ in Section 5, and we present our concluding remarks in Section 6.

\section{Ladder operators, difference and differential equations}
The ladder operator method due to Chen and Ismail \cite{Chen1997} is a very useful technique frequently employed for analyzing the recurrence coefficients of orthogonal polynomials as well as the related Hankel determinants. Recently, the first author and Fang \cite{MF} derived the ladder operators and compatibility conditions for orthogonal polynomials with the generalized Laguerre-type weight of the form
\be\label{gf}
w(x) = {x^\al }w_0(x),\qquad\qquad x\in\mathbb{R}^{+},\quad \al > -1,
\ee
where $w_0(x)$ is continuously differentiable on $[0,\infty)$ and all the moments of $w(x)$ exist. (Chen and Ismail's ladder operator method requires $\al>0$.)

For our problem, the monic orthogonal polynomials $P_n(x)$ satisfy the ladder operator equations
\begin{equation}\label{bnn1}
\left(\frac{d}{dx}+B_{n}(x)\right)P_{n}(x)=\beta_{n}A_{n}(x)P_{n-1}(x),
\end{equation}
\begin{equation}\label{bn2}
\left(\frac{d}{dx}-B_{n}(x)-\mathrm{v}'(x)\right)P_{n-1}(x)=-A_{n-1}(x)P_{n}(x),
\end{equation}
where $\mathrm{v}(x):=-\ln w(x)$ is the potential and
\begin{equation}\label{an}
A_{n}(x):=\frac{1}{x}\cdot\frac{1}{h_{n}}\int_{0}^{\infty}\frac{x \mathrm{v}'(x)-y \mathrm{v}'(y)}{x-y}P_{n}^{2}(y)w(y)dy,
\end{equation}
\begin{equation}\label{bn}
B_{n}(x):=\frac{1}{x} \left( \dfrac{1}{h_{n-1}}\int_{0}^{\infty}\dfrac{x \mathrm{v}'(x)-y \mathrm{v}'(y)}{x-y}P_{n}(y)P_{n-1}(y)w(y)dy-n\right).
\end{equation}
We often do not display the $t$ or $\al$-dependence of many quantities for brevity. For example, we may write the recurrence coefficients $\al_n(t,\al)$ and $\bt_n(t,\al)$ as $\al_n(t)$ and $\bt_n(t)$, or even simply as $\al_n$ and $\bt_n$.
The functions $A_n(x)$ and $B_n(x)$ defined above are not independent but satisfy the following compatibility conditions \cite{MF}:
\be
B_{n+1}(x)+B_{n}(x)=(x-\alpha_n) A_{n}(x)-\mathrm{v}'(x), \tag{$S_{1}$}
\ee
\be
1+(x-\alpha_n)(B_{n+1}(x)-B_{n}(x))=\beta_{n+1}A_{n+1}(x)-\beta_{n}A_{n-1}(x), \tag{$S_{2}$}
\ee
\be
B_{n}^{2}(x)+\mathrm{v}'(x)B_{n}(x)+\sum_{j=0}^{n-1}A_{j}(x)=\beta_{n}A_{n}(x)A_{n-1}(x),\tag{$S_{2}'$}
\ee
where ($S_{2}'$) is obtained from a suitable combination of ($S_{1}$) and ($S_{2}$). Next, we will use the ladder operators and associated compatibility conditions to analyze our problem.

From (\ref{w1}) we see that
\begin{equation}\label{pt1}
\mathrm v(x) =  - \ln w(x) = x +\dfrac{t}{x}-\alpha \ln x
\end{equation}
and
\begin{equation}\label{vp1}
\frac{x \mathrm v^\prime(x) - y \mathrm v^\prime(y)}{x - y} = 1+ \frac{t }{x y}.
\end{equation}
Substituting (\ref{vp1}) into (\ref{an}) and (\ref{bn}), we find
$$
\begin{aligned}
A_{n}(x)&=\dfrac{1}{x}\cdot\dfrac{1}{h_{n}}\int_{0}^{\infty}\left(1+\dfrac{t }{xy}\right)P_{n}^{2}(y)w(y)d y\\[8pt]
& =\dfrac{1}{x}+\dfrac{t}{x^2 h_{n}}\int_{0}^{\infty}\dfrac{P_{n}^{2}(y)}{y }w(y)d y,
\end{aligned}
$$
$$
\begin{aligned}
B_{n}(x)&=\dfrac{1}{x}\left(\dfrac{1}{h_{n-1}}\int_0^\infty{\left(1+\dfrac{t }{xy}\right)P_{n}(y)P_{n-1}(y) w(y)}d y -n\right)\\[8pt]
&=-\dfrac{n}{x}+ \dfrac{t }{{x^2 h_{n - 1}}}\int_0^\infty  {\dfrac{{{P_n}(y){P_{n - 1}}(y)}}{{y }}} w(y)d y.
\end{aligned}
$$
Hence, we have the following lemma.
\begin{lemma}
The functions ${A_n}(x)$ and ${B_n}(x)$ can be expressed as follows:
\begin{equation}\label{anz1}
{A_n}(x) =
\frac{1}{x} +\frac{{R_n}(t)}{x^2},
\end{equation}
\begin{equation}\label{bnz1}
{B_n}(x) = -\dfrac{n}{x} + \dfrac{{r_n}(t)}{x^2},
\end{equation}
where ${R_n}(t)$ and ${r_n}(t)$ are two auxiliary quantities given by
\begin{equation}\label{Rnz1}
{R_n}(t): = \frac{t}{{h_n}}\int_0^\infty  {\frac{{P^2_n}(y)}{y}}w(y)d y,
\end{equation}
\begin{equation}\label{rnz1}
{r_n}(t): = \frac{t}{{h_{n - 1}}}\int_0^\infty  {\frac{{{P_n}(y){P_{n - 1}}(y)}}{{y}}}w(y)d y.
\end{equation}
\end{lemma}
\begin{remark}
The expressions of ${A_n}(x)$ and ${B_n}(x)$ in the above lemma are coincident with \cite[Lemma 2]{CI2010}. It can be seen that our derivation is more straightforward and the results are valid for $\al>-1,\; t\geq 0$.
\end{remark}

Substituting (\ref{anz1}) and (\ref{bnz1}) into ($S_{1}$) and ($S_{2}'$), we obtain the following identities satisfied by the recurrence coefficients and the auxiliary quantities (see also \cite{CI2010}):
\begin{equation}\label{s11}
{R_n}(t)- {\alpha _n}(t)+2 n+1+\alpha=0,
\end{equation}
\begin{equation}\label{s12}
{r_n}(t) + {r_{n + 1}}(t) = t  - {\alpha _n}(t){R_n}(t),
\end{equation}
\begin{equation}\label{s21}
{r^2_n}{(t)} - t{r_n}(t) = {\beta _n}(t){R_n}(t){R_{n - 1}}(t),
\end{equation}
\begin{equation}\label{s23}
n t -(2n+\alpha) {r_n}(t) = {\beta _n}(t)({R_n}(t)+ {R_{n - 1}}(t)),
\end{equation}
\begin{equation}\label{s22a}
n(n+\alpha)+{r_n}(t)+ \sum\limits_{j = 0}^{n - 1} {{R_j}(t)}={\beta _n}(t).
\end{equation}
\begin{lemma}
The auxiliary quantities $R_n(t)$ and $r_n(t)$ are expressed in terms of the recurrence coefficients $\alpha_n(t)$ and $\beta_n(t)$ as follows:
\begin{equation}\label{alpha2}
 {R_n}(t)={\alpha _n}(t)-2 n-1 - \alpha,
\end{equation}
\begin{equation}\label{s25}
(2 n+\alpha  ){r_n}(t) = nt-\beta _n(t)({\alpha _n}(t)+{\alpha _{n-1}}(t)-4n-2\alpha).
\end{equation}
\end{lemma}
\begin{proof}
From (\ref{s11}) we get (\ref{alpha2}). Substituting (\ref{alpha2}) into (\ref{s23}) to eliminate $R_n(t)$ and $R_{n-1}(t)$, we obtain (\ref{s25}).
\end{proof}
\begin{theorem}
The orthogonal polynomials $P_n(x)$ satisfy the second-order differential equation
\begin{align}\label{p4}
&{P''_n}(x) - \bigg(\mathrm v^\prime(x)  + \dfrac{{A_n'(x)}}{{{A_n}(x)}}\bigg)P_n'(x) + \bigg( B_n'(x) - B_n^2(x) - \mathrm v^\prime(x) {B_n}(x)+ {\beta _n}{A_n}(x){A_{n - 1}}(x)\no\\[5pt]
&-\dfrac{{A_n'(x){B_n}(x)}}{{{A_n}(x)}}\bigg) {P_n}(x) = 0,
\end{align}
where $\mathrm v^\prime(x)=1-\frac{\al}{x}-\frac{t}{x^2}$ and
\begin{equation}\label{A1}
A_n(x)=\frac{1}{x}+\frac{{\alpha _n}(t)-2 n-1 - \alpha}{x^2},
\end{equation}
\begin{equation}\label{B1}
B_n(x)=- \dfrac{n}{x}+\frac{nt-\beta _n(t)({\alpha _n}(t)+{\alpha _{n-1}}(t)-4n-2\alpha)}{(2 n+\alpha)x^2}.
\end{equation}
\end{theorem}
\begin{proof}
Eliminating $P_{n-1}(x)$ from the ladder operator equations (\ref{bnn1}) and (\ref{bn2}), we obtain  (\ref{p4}). Inserting (\ref{alpha2}) and (\ref{s25}) into (\ref{anz1}) and (\ref{bnz1}), we get the results in  (\ref{A1}) and (\ref{B1}).
\end{proof}
\begin{remark}
In the case $t=0$, it can be seen from (\ref{Rnz1}) and (\ref{rnz1}) that $R_n(0)=0,\; r_n(0)=0$. By making use of (\ref{s11}) and (\ref{s22a}), we have $\al_n(0)=2n+1+\al,\;\bt_n(0)=n(n+\al)$. Then, the differential equation (\ref{p4}) is reduced to the differential equation satisfied by classical Laguerre polynomials \cite{G}
$$
xP_n''(x)+(\al+1-x)P_n'(x)+nP_n(x)=0.
$$
\end{remark}
\begin{theorem}\label{th1}
The recurrence coefficients $\alpha_n$ and $\bt_n$ satisfy a \textbf{first-order} system of difference equations as follows:
\begin{subequations}\label{ds1}
\begin{align}\label{d11}
&\al t+2\bt_n(\al_n+\al_{n-1})+(2n+\al)\left[\al_n(2n+2+\al-\al_n)-3\bt_n-\bt_{n+1}\right]=0,\\[8pt]
&\left[nt-\beta_n(\alpha_n+\alpha_{n-1}-4n-2\alpha)\right]\left[(n+\al)t+\beta_n(\alpha_n+\alpha_{n-1}-4n-2\alpha)\right]\no\\
&+(2n+\alpha)^2\beta_n(\alpha_n-2n-1-\alpha)(\alpha_{n-1}-2n+1-\alpha)=0.\label{d12}
\end{align}
\end{subequations}
\end{theorem}
\begin{proof}
To derive the first equation (\ref{d11}), we start from replacing $n$ by $n+1$ in (\ref{s22a}) and making a difference with (\ref{s22a}), which gives
$$
r_{n+1}-r_n+R_n+2n+1+\al=\bt_{n+1}-\bt_n.
$$
Using (\ref{s11}), it follows that
\be\label{de}
r_{n+1}-r_n+\al_n=\bt_{n+1}-\bt_n.
\ee
Eliminating $r_{n+1}$ from (\ref{s12}) and (\ref{de}), we have
$$
2r_n-\al_n=t-\al_nR_n+\bt_n-\bt_{n+1}.
$$
Multiplying both sides of the above equation by $2n+\alpha$, and substituting (\ref{alpha2}) and (\ref{s25}) into it, we obtain (\ref{d11}) after simplification.
Multiplying both sides of equation (\ref{s21}) by $(2n+\alpha)^2$ and using (\ref{alpha2}) and (\ref{s25}), we arrive at (\ref{d12}).
\end{proof}
\begin{remark}
One will get a \textbf{second-order} difference equation for the recurrence coefficients by substituting (\ref{alpha2}) and (\ref{s25}) into (\ref{s12}) directly. The discrete system (\ref{ds1}) is very important and will be used to derive the large $n$ asymptotic expansions of the recurrence coefficients in Section 4.
\end{remark}
\begin{remark}
We mention that Chen and Its \cite{CI2010} derived a system of difference equations for the auxiliary quantities $R_n$ and $r_n$ from (\ref{s11})--(\ref{s23}):
\begin{subequations}\label{cid}
\begin{align}
&r_n+ r_{n+1}=t-(2n+1+\al+R_n)R_n,\\
&(r_n^2-tr_n)(R_n+R_{n-1})=[nt-(2n+\al)r_n]R_nR_{n-1}.
\end{align}
\end{subequations}
Van Assche \cite[p. 58--59]{W} showed that the system (\ref{cid}) can be identified as alternate discrete Painlev\'{e} II equations and can be transformed into a discrete Painlev\'{e} equation with symmetry the affine Weyl group $(2A_1)^{(1)}$ and surface $D_6^{(1)}$ from \cite{KNY}. Our discrete system (\ref{ds1}) for the recurrence coefficients is equivalent to the discrete system (\ref{cid}) for the auxiliary quantities via the relations (\ref{s11})--(\ref{s22a}).
\end{remark}

Next, we show that the recurrence coefficients satisfy a Toda-type system.
\begin{theorem}
The recurrence coefficients $\alpha_n(t)$ and $\beta_n(t)$ satisfy a system of differential-difference equations
\begin{subequations}\label{tb}
\begin{align}
&t \alpha'_n(t)=\alpha_n(t)+\beta_{n}(t)-\beta_{n+1}(t),\label{tb2}\\
&t \beta'_n(t)=\beta_n(t)(\alpha_{n-1}(t)-\alpha_n(t)+2).\label{tb1}
\end{align}
\end{subequations}
\end{theorem}
\begin{proof}
From (\ref{h_j}) we have
$$
\int_{0}^{\infty}P_n(x;t)P_{n-1}(x;t)w(x;t) dx=0
$$
and
$$
\int_{0}^{\infty}P_n^2(x;t)w(x;t) dx=h_n(t).
$$
By taking derivatives with respect to $t$ on both sides respectively, we obtain
\begin{equation}\label{ddt1}
t\frac{d}{d t}\mathrm{p}(n,t)=r_n(t)
\end{equation}
and
\begin{equation}\label{dlnhnt}
t\frac{d}{d t}\ln h_n(t)=-R_n(t).
\end{equation}
From (\ref{alpha1}) and (\ref{ddt1}) we find
\be\label{equ}
t\alpha_n'(t)=r_n(t)-r_{n+1}(t).
\ee
The combination of (\ref{de}) and (\ref{equ}) gives the result in (\ref{tb2}).
Using (\ref{be}) and (\ref{dlnhnt}) we have
\be\label{bn3}
t \beta_{n}'(t)=\beta_n(t) (R_{n-1}(t)- R_{n}(t)).
\ee
Substituting (\ref{alpha2}) into (\ref{bn3}), we arrive at (\ref{tb1}).
\end{proof}
\begin{remark}
The Toda-type system (\ref{tb}) will play an important role in the derivation of the long-time asymptotics of the recurrence coefficients in Section 5.
\end{remark}

Let $H_n(t)$ be a quantity related to the logarithmic derivative of the Hankel determinant
\begin{equation}\label{Pin}
H_n(t): = t\frac{d}{{dt}}\ln {D_n}(t).
\end{equation}
In the following, we derive some relations between $H_n(t)$ and other quantities including the recurrence coefficients, the sub-leading coefficient and the auxiliary quantities.
By making use of (\ref{d_nt1}) and (\ref{dlnhnt}), we find
\begin{equation}\label{eq4}
H_n(t)=-\sum_{j=0}^{n-1}R_j(t).
\end{equation}
It follows that
\begin{equation}\label{tR}
R_n(t)=H_n(t)-H_{n+1}(t).
\end{equation}
Substituting (\ref{alpha2}) into (\ref{eq4}) and using (\ref{pn}), we obtain
\begin{equation}\label{r12}
H_n(t)=n(n + \alpha) + \mathrm p(n,t).
\end{equation}
On the other hand, from (\ref{s22a}) we have
\begin{equation}\label{r11}
H_n(t)=n(n + \alpha)+r_n(t)-\beta_n(t).
\end{equation}
The combination of  (\ref{r12}) and (\ref{r11}) yields
\begin{equation}\label{pp}
\mathrm{p}(n,t)={r_n}(t)- {\beta _n}(t) .
\end{equation}
The above relations will be used in the asymptotic analysis of our problem in Sections 4 and 5. Finally, we present the main results of Chen and Its \cite{CI2010} in the following theorem.
\begin{theorem}[Chen and Its \cite{CI2010}]\label{thm1}
The auxiliary quantity $R_n(t)$, related to the recurrence coefficient $\al_n(t)$ via (\ref{alpha2}), satisfies the Painlev\'{e} III\,$'$ equation \cite[(A.45.4\,$'$)]{Mehta}
$$
R_n''(t)=\frac{(R_n'(t))^2}{R_n(t)}-\frac{R_n'(t)}{t}+\frac{\tilde{\al} R_n^2(t)+\tilde{\gamma} R_n^3(t)}{4t^2}+\frac{\tilde{\beta}}{4t}+\frac{\tilde{\delta}}{4R_n(t)},
$$
with parameters $\tilde{\al}=4(2n+1+\al),\;\tilde{\gamma}=4,\;\tilde{\beta}=4\al,\;\tilde{\delta}=-4$.
The quantity $H_n(t)=t\frac{d}{{dt}}\ln {D_n}(t)$ satisfies the second-order nonlinear differential equation
\be\label{deh}
\left(tH_n''(t)\right)^2=\left[n-(2n+\al)H_n'(t)\right]^2-4\left[n(n+\al)+tH_n'(t)-H_n(t)\right]H_n'(t)(H_n'(t)-1).
\ee
\end{theorem}
\begin{remark}
Chen and Its \cite{CI2010} stated that equation (\ref{deh}) can be transformed into the Jimbo-Miwa-Okamoto $\s$-form of the Painlev\'{e} III equation. We give the details below.
Let
$$
H_n(t)=\frac{1}{2}\left[\s_n(s)+s^2+n(n+\al)\right],\qquad t=s^2\; (s\geq 0).
$$
Then, equation (\ref{deh}) is converted into the $\s$-form of the Painlev\'{e} III equation \cite[(C.29)]{Jimbo1981}
\begin{align}\label{jmo}
\left(s\s_n''(s)-\s_n'(s)\right)^2=&\:4\left(2\s_n(s)-s\s_n'(s)\right)\left((\s_n'(s))^2-4s^2\right)+2\left(\theta_{0}^2+\theta_{\infty}^2\right)\left((\s_n'(s))^2+4s^2\right)\no\\
&-16\theta_{0}\theta_{\infty}s\s_n'(s),
\end{align}
with parameters $\theta_{0}=\al,\;\theta_{\infty}=-2n-\al$ (the choice of $\theta_{0}$ and $\theta_{\infty}$ is not unique due to the symmetric form of (\ref{jmo})). The results in Theorem \ref{thm1} are valid for $\al>-1$ from our analysis.
\end{remark}

\section{Zeros of orthogonal polynomials with the singularly perturbed Laguerre weight}
It is known that the zeros of orthogonal polynomials are real, simple and are located in the interior of the supporting set of the associated distribution or weight function. The zeros also satisfy the interlacing property. Ismail and Li \cite{IL} utilized chain sequences to propose a method to determine the upper (lower) bounds for the largest (smallest) zeros of orthogonal polynomials in terms of their recurrence coefficients. The three-term recurrence relation and mixed recurrence relation for orthogonal polynomials can be used to find inner bounds for the extreme zeros of polynomials \cite{kd}.
\begin{lemma}
The monic orthogonal polynomials with respect to the singularly perturbed Laguerre weight (\ref{w1}) satisfy the following mixed three-term recurrence relation:
\begin{align}\label{xpp}
x^{2} P_{n-2}(x ; t, \alpha+2) =& \left[\frac{e_{n}}{\beta_{n-1}(t,\al)}\left(x-\alpha_{n-1}(t,\al)\right)-d_{n}\right] P_{n-1}(x ; t, \alpha)\no\\[5pt] &+\left(1-\frac{e_{n}}{\beta_{n-1}(t,\al)}\right) P_{n}(x ; t, \alpha),
\end{align}
where $\alpha_{n}(t,\al)$ and $\beta_{n}(t,\al)$ are the recurrence coefficients in (\ref{xp}) and
$$
d_{n}=\frac{P_{n}(0 ; t, \alpha)}{P_{n-1}(0 ; t, \alpha)}+\frac{P_{n-1}(0 ; t, \alpha+1)}{P_{n-2}(0 ; t, \alpha+1)},
$$
$$
e_{n}=\frac{P_{n-1}(0 ; t, \alpha)}{P_{n-2}(0 ; t, \alpha)}\cdot\frac{P_{n-1}(0 ; t, \alpha+1)}{P_{n-2}(0 ; t, \alpha+1)}.
$$
\end{lemma}
\begin{proof}
Our proof follows the approach of \cite[Lemma 3.6]{Cl}. We have for our weight
$$
w(x ; t, \alpha+2) = x^{\alpha+2} \mathrm e^{ - x-\frac{t}{x}} = x w(x ; t, \alpha+1).
$$
By using Christoffel's formula (see, e.g., \cite[Theorem 2.7.1]{Ismail}) and letting
$$
d\mu(x)=w(x ; t, \alpha+1)dx,\quad \Phi(x)=x,\quad S_n(x)=P_{n-2}(x;t,\al+2),
$$
we find
$$
{P_{n-2}(0 ; t, \alpha+1)}x P_{n-2}(x ; t, \alpha+2)=\begin{vmatrix}
P_{n-2}(0 ; t, \alpha+1) & P_{n-1}(0 ; t, \alpha+1)  \\
P_{n-2}(x ; t, \alpha+1) & P_{n-1}(x ; t, \alpha+1)
\end{vmatrix}.
$$
It follows that
\begin{equation} \label{pp3}
x^2P_{n-2}(x ; t, \alpha+2) = xP_{n-1}(x ; t, \alpha+1)-\frac{P_{n-1}(0 ; t, \alpha+1)}{P_{n-2}(0 ; t, \alpha+1)} xP_{n-2}(x ; t, \alpha+1).
\end{equation}
Similarly, we have
\begin{equation}\label{pp11}
xP_{n-1}(x ; t, \alpha+1) = P_{n}(x ; t, \alpha)-\frac{P_{n}(0 ; t, \alpha)}{P_{n-1}(0 ; t, \alpha)} P_{n-1}(x ; t, \alpha)
\end{equation}
and
\begin{equation}\label{pp2}
xP_{n-2}(x ; t, \alpha+1) = P_{n-1}(x ; t, \alpha)-\frac{P_{n-1}(0 ; t, \alpha)}{P_{n-2}(0 ; t, \alpha)} P_{n-2}(x ; t, \alpha).
\end{equation}
Substitute (\ref{pp11}) and
(\ref{pp2}) into (\ref{pp3}) and use the three-term recurrence relation (\ref{xp}) to eliminate $P_{n-2}(x ; t, \alpha)$, then $x^2P_{n-2}(x ; t, \alpha+2)$ can be expressed only in terms of $P_{n}(x ; t, \alpha)$ and $P_{n-1}(x ; t, \alpha)$ and  the result is given by (\ref{xpp}).
\end{proof}
\begin{remark}
It can be seen from the above proof that the mixed three-term recurrence relation (\ref{xpp}) (the $t$-dependence may not be displayed) holds for the monic orthogonal polynomials with respect to the general Laguerre-type weight of the form given in (\ref{gf}).
\end{remark}
Following the work on the zeros of semi-classical Laguerre polynomials and generalized Airy polynomials studied by Clarkson and Jordaan \cite{CJ2018,Cl}, we have the following theorem.
\begin{theorem}\label{tz}
Let $x_{1,n} < x_{2,n} <  \cdots  < x_{n,n}$ denote the $n$ zeros of the orthogonal polynomials $P_n(x;t,\al)$ with respect to the singularly perturbed Laguerre weight (\ref{w1}). Then the zeros
\begin{enumerate}
    \item[(i)] are real, distinct and
$$
    0 < x_{1,n} < x_{1,n - 1} < x_{2,n} < \cdots < x_{n - 1,n} < x_{n - 1,n - 1} < x_{n,n};
$$
    \item[(ii)] strictly increase with $\alpha$ and  strictly increase with  $t$ (for any fixed zero $x_{j,n},\;j=1,2,\ldots,n$);
    \item[(iii)] satisfy
    $$
    a_n < x_{1,n} < \al_{n-1} < x_{n,n} < b_n,
    $$
    where
$$
a_n = \min_{1 \leq k \leq n-1} \left\{ \frac{1}{2}\left(\alpha_k + \alpha_{k-1}\right) - \frac{1}{2}\sqrt{\left(\alpha_k - \alpha_{k-1}\right)^2 + 4c_n \beta_k} \right\},
$$
$$
b_n = \max_{1 \leq k \leq n-1} \left\{ \frac{1}{2}\left(\alpha_k + \alpha_{k-1}\right) + \frac{1}{2}\sqrt{\left(\alpha_k - \alpha_{k-1}\right)^2 + 4c_n \beta_k} \right\},
$$
with  $c_n = 4\cos^2\left( \frac{\pi}{n+1} \right) + \varepsilon,\; \varepsilon \geq 0$;
\item[(iv)]
satisfy
\be\label{bo}
0 < x_{1, n} < \alpha_{n-1} + \frac{d_{n} \beta_{n-1}}{e_{n}} < x_{n, n},
\ee
where
$$
d_{n} = \frac{P_{n}(0 ; t, \alpha)}{P_{n-1}(0 ; t, \alpha)} + \frac{P_{n-1}(0 ; t, \alpha+1)}{P_{n-2}(0 ; t, \alpha+1)},
$$
$$
e_{n} =\frac{P_{n-1}(0 ; t, \alpha)}{P_{n-2}(0 ; t, \alpha)}\cdot \frac{P_{n-1}(0 ; t, \alpha+1)}{P_{n-2}(0 ; t, \alpha+1)}.
$$
\end{enumerate}
In (iii) and (iv), $\alpha_{n}=\alpha_{n}(t,\al)$ and $\beta_{n}=\beta_{n}(t,\al)$ are the recurrence coefficients in (\ref{xp}).
\end{theorem}
\begin{proof}
\begin{enumerate}
\item[(i)] See Szeg\H{o} \cite[Theorems 3.3.1 and 3.3.2]{G}.
\item[(ii)] For our weight (\ref{w1}), it can be seen that
    $$
    \frac{\partial}{\partial t}\ln w(x;t,\al)=-\frac{1}{x}
    $$
    and
    $$
    \frac{\partial}{\partial \al}\ln w(x;t,\al)=\ln x
    $$
are both increasing functions of $x$ on $\mathbb{R}^{+}$. Using Markov's monotonicity theorem \cite[Theorem 6.12.1]{G}, we have that the $j$th zero $x_{j,n}$ (for fixed $j=1, 2, \ldots, n$) is an increasing function of both $t$ and $\al$.
\item[(iii)]
From the three-term recurrence relation (\ref{xp}) we have
$$
       \beta_{n-1}P_{n-2}(x;t,\alpha)=(x-\alpha_{n-1}(t,\al))P_{n-1}(x;t,\alpha)-P_{n}(x;t,\alpha),
$$
which corresponds to $k=1$ and $G_1(x)=x-\al_{n-1}(t,\al)$ in \cite[Equation (1)]{kd}. It follows from \cite[Corollary 2.2]{kd} that $\al_{n-1}(t,\al)$ is an inner bound of the extreme zeros $x_{1,n}$ and $x_{n,n}$. The outer bounds $a_n$ and $b_n$ for the extreme zeros come from Ismail \cite[Theorems 7.2.6 and 7.2.7]{Ismail} (see also Ismail and Li \cite[Theorems 2 and 3]{IL}) using chain sequences.
\item[(iv)]
    Using the mixed recurrence relation (\ref{xpp}) and following the similar analysis in \cite[Theorem 3.9]{Cl}, it can be found that the zero of $G(x):=\frac{e_{n}}{\beta_{n-1}}\left(x-\alpha_{n-1}\right)-d_{n}$ (the coefficient of $P_{n-1}(x)$ in (\ref{xpp})) and the $n-2$ zeros of $P_{n-2}(x ; t, \alpha+2)$ interlace with the $n$ zeros of $P_{n}(x ; t, \alpha)$. It follows that the zero $\alpha_{n-1} + \frac{d_{n} \beta_{n-1}}{e_{n}}$  of $G(x)$ must lie between the two extreme zeros of $P_{n}(x ; t, \alpha)$. Hence, we obtain another inner bound for the extreme zeros in (\ref{bo}).\qedhere
\end{enumerate}
\end{proof}
\begin{remark}
From the proof of the above theorem, it can be found that Theorem \ref{tz}(iii) holds for general orthogonal polynomials with a three-term recurrence relation and Theorem \ref{tz}(iv) holds for orthogonal polynomials with respect to the general Laguerre-type weight (\ref{gf}).
\end{remark}
\begin{remark}
If $t=0$, our orthogonal polynomials $P_n(x;t,\al)$ are reduced to the classical (monic) Laguerre polynomials and $\al_n(0,\al)=2n+1+\al,\;\bt_n(0,\al)=n(n+\al)$. In this case, Theorem \ref{tz}(iii) coincides with the result for classical Laguerre polynomials \cite[Theorem 7.2.8]{Ismail} by choosing $c_n=4$ for $a_n$ and $c_n=4\cos^2\left( \frac{\pi}{n+1} \right)$ for $b_n$.
\end{remark}
\begin{lemma}
Let
\be\label{qp}
Q_n(x):=\sqrt{\frac{w(x)}{A_n(x)}}\:P_n(x).
\ee
Then $Q_n(x)$ satisfies the second-order differential equation
\be\label{ode}
Q_n''(x)+F(x)Q_n(x)=0,
\ee
where
\begin{align}\label{fx}
F(x):=&\:\bt_nA_n(x)A_{n-1}(x)-\frac{3(A_n'(x))^2}{4A_n^2(x)}+\frac{A_n''(x)-A_n'(x)\mathrm{v}'(x)-2A_n'(x)B_n(x)}{2A_n(x)}\no\\[8pt]
&-B_n^2(x)+B_n'(x)-B_n(x)\mathrm{v}'(x)-\frac{(\mathrm{v}'(x))^2}{4}+\frac{\mathrm{v}''(x)}{2},
\end{align}
and $A_n(x), B_n(x)$ and $\mathrm{v}(x)$ are given by (\ref{A1}), (\ref{B1}) and (\ref{pt1}), respectively.
\end{lemma}
\begin{proof}
The result is obtained by substituting $P_n(x)=Q_n(x)\sqrt{A_n(x)/w(x)}$ and $w(x)=\mathrm{e}^{-\mathrm{v}(x)}$ into the second-order differential equation (\ref{p4}).
\end{proof}
\begin{theorem}
Let $F(x)$ be the function in (\ref{fx}) and $y_1<y_{2}<y_{3}<\cdots$ be the zeros of $P_n(x)$ in an interval $(c,d)\subset \mathbb{R}^{+}$. Then
\begin{enumerate}
    \item[(i)] if there exists $M_1>0$ such that $F(x)<M_1$ on $(c,d)$, then
    $$
    \Delta y_k\equiv y_{k+1}-y_{k}>\frac{\pi}{\sqrt{M_1}},\qquad k=1,2,\ldots;
    $$
    \item[(ii)] if there exists $M_2>0$ such that $F(x)>M_2$ on $(c,d)$, then
    $$
    \Delta y_k\equiv y_{k+1}-y_{k}<\frac{\pi}{\sqrt{M_2}},\qquad k=1,2,\ldots;
    $$
    \item[(iii)] if $F(x)$ is strictly increasing on $(c,d)$, then $\Delta^2 y_k\equiv y_{k+2}-2y_{k+1}+y_{k}<0\; (y_{k+2}-y_{k+1}<y_{k+1}-y_{k}),\;k=1,2,\ldots$, i.e., the zeros of $P_n(x)$ in $(c,d)$ are concave;
    \item[(iv)] if $F(x)$ is strictly decreasing on $(c,d)$, then $\Delta^2 y_k\equiv y_{k+2}-2y_{k+1}+y_{k}>0\; (y_{k+2}-y_{k+1}>y_{k+1}-y_{k}),\;k=1,2,\ldots$, i.e., the zeros of $P_n(x)$ in $(c,d)$ are convex.
\end{enumerate}
\end{theorem}
\begin{proof}
By (\ref{w1}) and (\ref{anz1}) we see that $w(x)$ and $A_n(x)$ are both strictly positive on $\mathbb{R}^{+}$. It follows from (\ref{qp}) that $P_n(x)$ and $Q_n(x)$ have the same zeros in $\mathbb{R}^{+}$. Applying Sturm's comparison and convexity theorems (see, e.g., \cite{DGS,JT}) to the second-order differential equation (\ref{ode}), we establish the theorem.
\end{proof}

\section{Large $n$ asymptotics}
For $\alpha > -1,\; t \ge  0$, consider the singularly perturbed Laguerre unitary ensemble on the space of $n\times n$ positive definite Hermitian matrices $M=(M_{ij})_{n\times n}$ with the probability distribution
$$
\frac{1}{Z_n(t)}(\det M)^{\al}\mathrm{e}^{-\mathrm{tr}\:V_{t}(M)}dM,\qquad dM=\prod_{i=1}^{n}dM_{ii}\prod_{1\leq i< j\leq n}d\mathfrak{R}M_{ij}d\mathfrak{I}M_{ij},
$$
where $Z_n(t)$ is the normalization constant or partition function, and $V_{t}(x)=x+\frac{t}{x}$. This distribution is invariant under every unitary transformation and gives rise to a probability distribution on the eigenvalues $x_1, x_2, \ldots, x_n$ of $M$ in the form
	$$
	P\left(x_1, x_2, \ldots, x_n\right) \prod_{j=1}^n d x_j=\frac{1}{Z_n(t)} \prod_{1 \leq i<j \leq n}(x_j-x_i)^2 \prod_{k=1}^n w(x_k;t) d x_k,\qquad x_1, x_2, \ldots, x_n\in \mathbb{R}^{+},
	$$
	where $ w(x;t)$ is the singularly perturbed Laguerre weight in (\ref{w1}). According to (\ref{mi}), there is a simple relation between the Hankel determinant $D_n(t)$ and the partition function $Z_n(t)$:
$$
D_n(t)=\frac{1}{n!}Z_n(t).
$$
For more information about the topic of random matrix theory, see \cite{Deift,Forrester2010,Mehta}.

In the framework of Dyson's Coulomb fluid approach \cite{Dyson}, the eigenvalues of the Hermitian matrices from the unitary ensemble can be approximated by a continuous fluid with an equilibrium density when $n$ is sufficiently large. For our problem, it can be seen that the potential $\mathrm{v}(x)$ in (\ref{pt1}) satisfies the condition that $x\mathrm{v}'(x)$ is increasing on $\mathbb{R}^{+}$. This leads to the fact that the support of the equilibrium density is a single interval, denoted by $(a, b),\;a>0$; see \cite[p. 199]{Saff}.

As stated in Chen and Ismail's work \cite{ci1997}, the equilibrium density $\sigma(\cdot)$ can be derived from a constraint minimization problem:
\begin{equation}\label{fe}
F[\s]:=\int_{a}^{b}\s(x)\mathrm{v}(x)dx-\int_{a}^{b}\int_{a}^{b}\s(x)\ln|x-y|\s(y)dxdy
\end{equation}
subject to the normalization condition
\be\label{nc}
\int_{a}^{b}\s(x)dx=n.
\ee
Upon minimization, the equilibrium density $\sigma(x)$ is found to satisfy the integral equation
\begin{equation}\label{ie}
\mathrm{v}(x)-2\int_{a}^{b}\ln|x-y|\s(y)d y=A,\qquad x\in (a,b),
\end{equation}
where $A$ is the Lagrange multiplier, a constant independent of $x$.
By differentiating with respect to $x$, equation (\ref{ie}) is converted into the singular integral equation
\begin{equation}\label{sie}
\mathrm{v}'(x)-2 P\int_{a}^{b}\frac{\sigma(y)}{x-y}d y=0,\qquad x\in (a,b),
\end{equation}
where $P$ denotes the Cauchy principal value. From (\ref{sie}) and with the aid of (\ref{nc}), it can be found that the endpoints $a$ and $b$ are determined by two supplementary conditions
\begin{equation}\label{sup1}
\int_{a}^{b}\frac{\mathrm{v}'(x)}{\sqrt{(b-x)(x-a)}}d  x=0,
\end{equation}
\begin{equation}\label{sup2}
\int_{a}^{b}\frac{x\:\mathrm{v}'(x)}{\sqrt{(b-x)(x-a)}}d  x=2 \pi n.
\end{equation}

Taking a partial derivative with respect to $n$ in (\ref{fe}) and making use of (\ref{nc}) and (\ref{ie}), we get the relation \cite[(2.14)]{ci1997}
\begin{equation}\label{pd}
\frac{\partial F[\s]}{\partial n}=A.
\end{equation}
It was also demonstrated in \cite{ci1997}  that as $n\rightarrow\infty$, the recurrence coefficients have the following asymptotic behavior:
\begin{subequations}\label{eqs}
\begin{align}
&\alpha_n=\frac{a+b}{2}+O\left(\frac{\partial^2 A}{\partial t \partial n}\right),\\[5pt]
&\beta_n=\left(\frac{b-a}{4}\right)^2\left(1+O\left(\frac{\partial^3 A}{\partial n^3}\right)\right).
\end{align}
\end{subequations}
Furthermore, Chen and Lawrence \cite{CL1998} showed that the monic orthogonal polynomials $P_n(z;t,\al)$ have the large $n$ asymptotic behavior
\be\label{pnz}
P_n(z;t,\al)\sim \exp(-S_1(z)-S_2(z)),\qquad z\in \mathbb{C}\backslash[a,b],
\ee
where
$$
S_1(z)=\frac{1}{4}\ln\left[\frac{16(z-a)(z-b)}{(b-a)^2}\left(\frac{\sqrt{z-a}-\sqrt{z-b}}{\sqrt{z-a}+\sqrt{z-b}}\right)^2\right],
$$
$$
S_2(z)=-n\ln\left(\frac{\sqrt{z-a}+\sqrt{z-b}}{2}\right)^2+\frac{1}{2\pi}\int_{a}^{b}\frac{\mathrm{v}(x)}{\sqrt{(b-x)(x-a)}}\left[\frac{\sqrt{(z-a)(z-b)}}{x-z}+1\right]dx.
$$

Substituting (\ref{pt1}) into (\ref{sup1}) and (\ref{sup2}) respectively, we obtain a system of algebraic equations satisfied by $X:=\frac{a+b}{2}$ (arithmetic mean) and $Y:=\sqrt{ab}$ (geometric mean) as follows:
$$tX+\al Y^2-Y^3=0,$$
\be\label{sys2}
X - {\frac{t}{{Y}}}=  2n +\alpha.
\ee
Eliminating $X$ from the above system gives rise to a quartic equation
    \begin{equation}\label{Y}
        Y^4-\alpha Y^3-(2n+\alpha)tY-t^2=0.
    \end{equation}
Using Mathematica, we find that equation (\ref{Y}) has a unique solution under the condition $Y>0$, and it has the series expansion as $n\rightarrow\infty$
\begin{align}\label{Y1}
Y =&\:\kappa  t^{1/3}n^{1/3}+\frac{\alpha }{3}+\frac{\alpha ^2 }{9 \kappa t^{1/3}n^{1/3}}+\frac{\alpha \left(2 \alpha ^2+27 t\right)}{81\kappa^2t^{2/3}n^{2/3}}+\frac{t}{6 n}-\frac{ \alpha  \left(2 \alpha ^4+27 \alpha ^2 t+81 t^2\right)}{1458\kappa t^{4/3} n^{4/3}}\no\\[8pt]
&-\frac{\alpha ^2 \left(7 \alpha ^4+108 \alpha ^2 t+972 t^2\right)}{13122\kappa^2 t^{5/3}n^{5/3}}-\frac{\alpha  t}{12 n^2}+ O(n^{ - {7/3}}),
\end{align}
where $\kappa= \sqrt[3]{2}$.
It follows from (\ref{sys2}) that
\begin{align}\label{b2}
X=\dfrac{a+b}{2}=&\:2n + \alpha  + \frac{t^{2/3}}{\kappa n^{1/3}} -\frac{ \alpha  t^{1/3}}{3 \kappa^2{n}^{2/3}}+\frac{\al( \alpha ^2-27  t)}{162 \kappa t^{1/3}n^{4/3}}+\frac{ \alpha ^4+54 \alpha ^2 t-81 t^2}{486\kappa^2 t^{2/3}n^{5/3}}\no\\[8pt]
&+\frac{\alpha  t}{12 n^2}+O(n^{ - 7/3})
\end{align}
and
\begin{align}\label{b4}
\left(\dfrac{b-a}{4}\right)^2=&\:\frac{X^2-Y^2}{4}\no\\
=&\:n^2+ \alpha n+\frac{ t^{2/3}n^{2/3}}{2 \kappa}-\frac{ \kappa\alpha   t^{1/3}n^{1/3}}{3} +\frac{\alpha ^2}{6}+\frac{\alpha \left(27 t-4 \alpha ^2\right)}{162 \kappa t^{1/3}n^{1/3}}\no\\[8pt]
&-\frac{5 \alpha ^4+108 \alpha ^2 t+81 t^2}{972\kappa^2 t^{2/3}n^{2/3}}+\frac{\alpha ^2  \left(7 \alpha ^4+108 \alpha ^2 t-243 t^2\right)}{26244\kappa t^{4/3}n^{4/3}}\no\\[8pt]
&+\frac{\alpha \left(8 \alpha ^6+135 \alpha ^4 t+1620 \alpha ^2 t^2+2187 t^3\right)}{78732\kappa^2 t^{5/3}n^{5/3}}+\frac{t^2}{48 n^2}+O(n^{ -{7/3}}).
\end{align}

Following the similar computations in \cite[p. 406--407]{Min2021}, multiplying by $\frac{1}{\sqrt{(b-x)(x-a)}}$ on both sides of (\ref{ie}) and integrating with respect to $x$ from $a$ to $b$,  we have
$$
\begin{aligned}
A =&\: \dfrac{1}{\pi }\int_a^b {\dfrac{\mathrm{v}(x)}{\sqrt {(b - x)(x-a)}}} dx - 2n\ln\frac{b-a}{4}\\
            =&\: \frac{a+b}{2}+\frac{t}{\sqrt{ab}}-2\alpha\ln \frac{\sqrt{a}+\sqrt{b}}{2}- 2 n\ln \dfrac{b-a}{4}\\
            =&\:X+\frac{t}{Y}-  n\ln \frac{X-Y}{2}-(n+\al)\ln\frac{X+Y}{2}.
\end{aligned}
$$
Substituting (\ref{Y1}) and (\ref{b2}) into the above, we find as $n\rightarrow\infty$
\begin{align}\label{a}
A =&- 2n \ln n + 2n- \alpha \ln n+\frac{3  t^{2/3}}{2\kappa n^{1/3}}-\frac{\alpha   t^{1/3}}{\kappa^2 n^{2/3}}-\frac{\alpha ^2}{3 n}-\frac{\al(2 \alpha ^2+27  t)}{108\kappa t^{1/3}n^{4/3}}\no\\[8pt]
&-\frac{ 2 \alpha ^4-216 \alpha ^2 t+81 t^2}{648\kappa^2 t^{2/3} n^{5/3}}+\frac{\alpha  (\alpha ^2+t)}{12 n^2} + O(n^{-7/3}).
\end{align}

Using the above results, we are able to derive the asymptotic expansions of the recurrence coefficients as $n\rightarrow\infty$.
\begin{theorem}\label{rc}
For fixed $t>0$, the recurrence coefficients $\alpha_n(t)$ and $\beta_{n}(t)$ have the following large $n$ asymptotic expansions:
\begin{subequations}\label{aba}
\begin{align}
\alpha_n(t)=&\:2 n+\alpha+1 +\frac{t^{2/3}}{\kappa n^{1/3}}-\frac{\alpha  t^{1/3}}{3\kappa^2 n^{2/3}}+\frac{(\alpha +1)  [\alpha(\alpha -1) -27 t]}{162 \kappa t^{1/3}n^{4/3}}\no\\[8pt]
&+\frac{\alpha ^2(\alpha ^2-1)+54 \alpha(\alpha +1) t-81 t^2}{486\kappa^2 t^{2/3}n^{5/3}}+\frac{\alpha  \left(\alpha ^2+81 t^2-1\right)}{972  tn^2}+O(n^{-7/3}),\label{aa}\\[8pt]
\beta_n(t)=&\:n^2 +\alpha n+\frac{t^{2/3} n^{2/3} }{2 \kappa}-\frac{\kappa \alpha   t^{1/3}n^{1/3}}{3} +\frac{6 \alpha ^2-1}{36} -\frac{\al(4 \alpha ^2-27  t-4 )}{162\kappa t^{1/3}n^{1/3}}\no\\[8pt]
&-\frac{5 \alpha ^4+\alpha ^2 (108 t-5)+81 t^2}{972\kappa^2 t^{2/3} n^{2/3}}-\frac{\alpha (\alpha ^2-1)}{486  tn}+O(n^{-4/3}),\label{ab}
\end{align}
\end{subequations}
where $\kappa= \sqrt[3]{2}$.
\end{theorem}
\begin{proof}
In view of (\ref{eqs}), (\ref{b2}),  (\ref{b4}) and (\ref{a}), it can be seen that $\al_n$ and $\bt_{n}$ have the large $n$ expansion forms
\begin{subequations}\label{abas}
\begin{align}
&\alpha _n = 2n+a_0+\sum _{j=1}^{\infty} \frac{a_j}{n^{j/3}},\\
&{\beta _n} =n^2+b_{-3}n+\sum _{j=-2}^{\infty} \frac{b_j}{n^{j/3}},
\end{align}
\end{subequations}
where $a_j, j=0, 1, \ldots$ and $b_j, j=-3, -2, \ldots$ are the expansion coefficients to be determined. Substituting (\ref{abas}) into the discrete system for the recurrence coefficients in (\ref{ds1}) and taking a large $n$ limit, we obtain the expansion coefficients $a_j$ and  $b_j$ recursively by letting all coefficients of powers of $n$ be zero. For example, $a_0$ and $b_{-3}$ satisfy the system of equations
\begin{align}
&a_0^2 -2 a_0 \left(b_{-3}+1\right)+(2 \alpha +3) b_{-3}-\alpha ^2-\alpha +1=0,\no\\
&\left(a_0-\alpha-1\right)^2 \left(b_{-3}-\alpha\right)=0,\no
\end{align}
which gives the unique solution
$$
a_0= \alpha +1,\qquad\qquad {b_{ - 3}} = \alpha.
$$
The other expansion coefficients are given by
$$
\begin{aligned}
&a_{1}=\frac{t^{2/3}}{\kappa},\qquad b_{-2}=\frac{t^{2/3}}{2 \kappa};\qquad a_{2}= -\frac{\alpha  t^{1/3}}{3\kappa^2 },\qquad b_{-1}=-\frac{\kappa \alpha   t^{1/3}}{3};\\[8pt]
&a_{3}= 0,\qquad b_{0}=\frac{6 \alpha ^2-1}{36} ;\qquad a_{4}=\frac{(\alpha +1)  [\alpha(\alpha -1) -27 t]}{162 \kappa t^{1/3}},\qquad b_{1}=-\frac{\al(4 \alpha ^2-27  t-4 )}{162\kappa t^{1/3}};\\[8pt]
&a_{5}=\frac{\alpha ^2(\alpha ^2-1)+54 \alpha(\alpha +1) t-81 t^2}{486\kappa^2 t^{2/3}},\qquad b_{2}=-\frac{5 \alpha ^4+\alpha ^2 (108 t-5)+81 t^2}{972\kappa^2 t^{2/3} };\\[8pt]
&a_{6}=\frac{\alpha  \left(\alpha ^2+81 t^2-1\right)}{972  t},\qquad b_{3}=-\frac{\alpha (\alpha ^2-1)}{486  t}
\end{aligned}
$$
and so on. The theorem is then established.
\end{proof}
In the following, we derive the large $n$ asymptotic expansions of the sub-leading coefficient $\mathrm p(n,t)$, the quantity $H_n(t) = t\frac{d}{{dt}}\ln {D_n}(t)$, the Hankel determinant $D_n(t)$ and the normalized constant $h_n(t)$ for fixed $t>0$ on the basis of the large $n$ asymptotic expansions of the recurrence coefficients obtained in Theorem \ref{rc}.

\begin{theorem}
The sub-leading coefficient of the monic orthogonal polynomials, $\mathrm p(n,t)$, has the large $n$ asymptotic expansion
\begin{align}\label{p1}
\mathrm{p}(n,t)=&- {n^2} - \alpha n-\frac{3 t^{2/3}n^{2/3} }{2 \kappa}+\frac{\alpha   t^{1/3}n^{1/3}}{\kappa^2} -\frac{6 \alpha ^2-18 t-1}{36} +\frac{\alpha  \left(\alpha ^2-27 t-1\right)}{54 \kappa t^{1/3}n^{1/3}}\no\\[8pt]
&+\frac{ \alpha ^4+\alpha ^2 (54 t-1)-81 t^2}{324\kappa^2 t^{2/3}n^{2/3}}+\frac{\alpha  \left(\alpha ^2+81 t^2-1\right)}{972 t n}+ O( n^ {- 4/3}),
\end{align}
where  $\kappa= \sqrt[3]{2}$.
\end{theorem}
\begin{proof}
From (\ref{s25}) and (\ref{pp}), we express $\mathrm{p}(n,t)$ in terms of the recurrence coefficients
\be\label{pnt}
\mathrm{p}(n,t)=\frac{nt-\left(\al_n(t)+\al_{n-1}(t)-2n-\al\right)\bt_n(t)}{2n+\al}.
\ee
Substituting (\ref{aba}) into the above and taking a large $n$ limit, we obtain the desired result.
\end{proof}
\begin{theorem}\label{hnt}
The quantity $H_n(t) = t\frac{d}{{dt}}\ln {D_n}(t)$ has the large $n$ asymptotic expansion
\begin{align}\label{dd}
H_n({t})=&-\frac{3 t^{2/3}n^{2/3} }{2 \kappa}+\frac{\alpha   t^{1/3}n^{1/3}}{\kappa^2} -\frac{6 \alpha ^2-18 t-1}{36} +\frac{\alpha  \left(\alpha ^2-27 t-1\right)}{54 \kappa t^{1/3}n^{1/3}}\no\\[8pt]
&+\frac{ \alpha ^4+\alpha ^2 (54 t-1)-81 t^2}{324\kappa^2 t^{2/3}n^{2/3}}+\frac{\alpha  \left(\alpha ^2+81 t^2-1\right)}{972 t n}+ O( n^ {- 4/3}),
\end{align}
where  $\kappa= \sqrt[3]{2}$.
\end{theorem}
\begin{proof}
The result is from the combination of (\ref{r12}) and (\ref{p1}).
\end{proof}
\begin{remark}
Our results for the large $n$ asymptotic expansions of the recurrence coefficients and the quantity $H_n(t)$ in Theorems \ref{rc} and \ref{hnt} coincide with those in the work of Xu, Dai and Zhao \cite[Corollary 1]{Xu2} using the Deift-Zhou steepest descent method for Riemann-Hilbert problems. Our method has the advantage of finding as many terms as one wants for the expansions. Furthermore, it can be seen that the large $n$ asymptotic expansions of the recurrence coefficients and the quantity $H_n(t)= t\frac{d}{{dt}}\ln {D_n}(t)$ are singular at the origin, which implies that the large $n$ asymptotic expansion of the Hankel determinant $D_n(t)$ can not be obtained directly by integrating the large $n$ asymptotic expansion of $H_n(t)/t$ from $0$ to $t$. It was pointed out by Mezzadri and Mo \cite[Section 1.1]{MM} that computing asymptotic formulae of Hankel determinants is a very important task---often a very difficult one---in many branches of mathematics and physics. The asymptotics of Hankel determinants depend crucially on the analytic properties of the weight $w(x)$. Usually, singular weights are the most challenging.
\end{remark}

Xu, Dai and Zhao \cite[Theorem 1]{Xu2} proved that the Hankel determinant $D_n(t)$ has the large $n$ asymptotic expansion
$$
\ln D_n(t)-\ln D_n(0)=\left(1+O(n^{-1/3})\right)\int_{0}^{t}\frac{1-4\al^2-8r(2n\xi)}{16\xi}d\xi,
$$
where the error term is uniformly valid for $t\in (0, d]$ and $d$ is a fixed positive constant. Here, $r(s)$ is a particular solution to the third-order nonlinear differential equation
$$
2s^2r'r'''-s^2(r'')^2+2sr'r''-4s(r')^3+\left(2r-\frac{1}{4}\right)(r')^2+1=0
$$
and this equation can be reduced to a particular Painlev\'{e} III equation. Moreover, $r(s)$ is analytic for $s\in(0,+\infty)$ and satisfies the following boundary conditions
$$
r(0)=\frac{1-4\al^2}{8}
$$
and
\be\label{rs}
r(s)=\frac{3}{2}s^{2/3}-\al s^{1/3}+O(1),\qquad s\rightarrow+\infty.
\ee
For fixed $t\in (0, d]$, note that
\begin{align}
\int_{0}^{t}\frac{1-4\al^2-8r(2n\xi)}{16\xi}d\xi&=\int_{0}^{2nt}\frac{1-4\al^2-8r(s)}{16s}ds\no\\
&=\int_{0}^{1}\frac{1-4\al^2-8r(s)}{16s}ds+\int_{1}^{2nt}\frac{1-4\al^2-8r(s)}{16s}ds.\no
\end{align}
By making use of (\ref{rs}), we have
$$
\int_{0}^{t}\frac{1-4\al^2-8r(2n\xi)}{16\xi}d\xi=-\frac{9}{8}(2nt)^{2/3}+\frac{3}{2}\al(2nt)^{1/3}+O(\ln n),\qquad n\rightarrow\infty.
$$

It is known that \cite[p. 321]{Mehta}
\begin{align}
D_n(0)&=\det\left(\int_{0}^{\infty}x^{i+j}{x^\alpha }{\mathrm e^{ - x}}dx\right)_{i,j=0}^{n-1}\no\\
&=\frac{G(n+1)G(n+\al+1)}{G(\al+1)}\no
\end{align}
and as $n\rightarrow\infty$
\begin{align}
\ln D_n (0)= &\:{n^2}\ln n - \dfrac{3}{2}{n^2}+ \alpha n \ln n- \left(\al-\ln (2\pi)\right)n+ \dfrac{3\alpha ^2- 1}{6}\ln n\no\\
&+ \dfrac{\alpha}{2}\ln (2\pi)+ 2 \zeta'(-1)- \ln G( \alpha+1 ) +O(n^{-1}),\no
\end{align}
where $G(z)$ is the Barnes $G$-function that satisfies the relation \cite{ew,Voros}
$$
G(z+1)=\Gamma(z)G(z),\qquad G(1)=1
$$
and $\zeta'(\cdot)$ is the derivative of the Riemann zeta function.
It follows that as $n\rightarrow\infty$
\be\label{dn1}
\ln D_n (t)= {n^2}\ln n - \dfrac{3}{2}{n^2}+ \alpha n \ln n- \left(\al-\ln (2\pi)\right)n -\frac{9}{8}(2nt)^{2/3}+O(n^{1/3})
\ee
for fixed $t\in (0, d]$. Hence, Xu, Dai and Zhao's result \cite{Xu2} only yields a large $n$ series expansion for the Hankel determinant up to order $n^{2/3}$. We try to derive the full asymptotic expansion of the Hankel determinant $D_n(t)$ in the following.
\begin{theorem}\label{thm}
For fixed $t>0$, the Hankel determinant $D_n(t)$ has the large $n$ asymptotic expansion
\begin{align}\label{dnt1}
{\ln D_n}(t) =& \:  {n^2}\ln n + \alpha  n\ln n+ \dfrac{12 \alpha ^2-5}{36}\ln n- \dfrac{3n^2}{2} - \left(\al-\ln (2\pi)\right)n-\frac{9 t^{2/3}n^{2/3} }{4 \kappa}\no\\[8pt]
&+\frac{3 \alpha t^{1/3}n^{1/3}}{\kappa^2}-\left(\frac{6 \alpha ^2-1}{36}\ln t-\frac{t}{2}+ \frac{\al^2\ln 2}{6}+\hat{c}_0(\al)\right)-\frac{\alpha  (2 \alpha ^2+27 t-2)}{36 \kappa t^{1/3}n^{1/3}}\no\\[8pt]
&-\frac{2 \alpha ^4-2 \alpha ^2 (108 t+1)+81 t^2}{432\kappa^2  t^{2/3}n^{2/3}}+\frac{\alpha  \left[2(81 t-1) \alpha ^2+162 t^2-135 t+2\right]}{1944 t n}\no\\[8pt]
&+O(n^{-4/3}),
\end{align}
where $\kappa= \sqrt[3]{2}$, and $\hat{c}_0(\al)$ is a constant depending only on $\alpha$ with a period of $1$ (as a function of $\al$).
\end{theorem}
\begin{proof}
From (\ref{pd}) and (\ref{a}), we see that the free energy $F[\s]$ has the large $n$ asymptotic expansion
\begin{align}\label{F}
F[\s]=&-n^2 \ln n-\alpha n \ln n-\dfrac{\alpha ^2 \ln n}{3}  +\dfrac{3 n^2}{2} +\alpha  n  +\frac{9 t^{2/3}n^{2/3}}{4 \kappa}-\frac{3  \alpha  t^{1/3}n^{1/3}}{\kappa^2}+C\no\\[8pt]
&+\frac{\al(2 \alpha ^2+27  t)}{36 \kappa t^{1/3}n^{1/3}}+\frac{ 2 \alpha ^4-216 \alpha ^2 t+81 t^2}{432\kappa^2 t^{2/3}n^{2/3}}-\frac{\alpha  \left(\alpha ^2+t\right)}{12 n}+O(n^{-4/3}),
\end{align}
where  $\kappa= \sqrt[3]{2}$ and $C$ is an integration constant independent of $n$.

Define  the ``free energy'' as
$$
F_n(t):=-\ln D_n(t).
$$
For sufficiently large $n$, as demonstrated by Chen and Ismail \cite{ci1997},
$F_n(t)$ is approximated by the free energy $F[\s]$ in (\ref{fe}) and the approximation is very accurate and effective. See also \cite{ChenIsmail1998} for applications of this method to three different examples. In the light of (\ref{F}), we assume that $F_n(t)$ has the large $n$ asymptotic expansion in the form
\be\label{fna}
{F_n}(t) = {c_9}(t,\al ){n^2}\ln n + {c_8}(t,\al )n\ln n+ {c_7}(t,\al )\ln n+\sum_{j=-3}^{6}c_j(t,\al )n^{j/3}+O(n^{-4/3}),
\ee
where $c_j(t,\al ), j=9, 8, \ldots, -3$ are the expansion coefficients to be determined.

From (\ref{bd}) we have the identity
$$
\ln\bt_n=2 F_n(t)-F_{n+1}(t)-F_{n-1}(t).
$$
Substituting (\ref{ab}) and (\ref{fna}) into the above and letting $n\rightarrow\infty$, we obtain the expansion coefficients $c_j$ (except $c_3$ and $c_0$) by equating coefficients of powers of $n$ on both sides.
The large $n$ asymptotic expansion for $F_n(t)$ reads
\begin{align}
{F_n}(t) =&  - {n^2}\ln n - \alpha  n\ln n+ \dfrac{5-12 \alpha ^2}{36}\ln n+ \dfrac{3n^2}{2} + {c_3}(t,\al)n+\frac{9 t^{2/3}n^{2/3} }{4 \kappa}\no\\[8pt]
&-\frac{3 \alpha t^{1/3}n^{1/3}}{\kappa^2}+ {c_0}(t,\al )+\frac{\alpha  (2 \alpha ^2+27 t-2)}{36 \kappa t^{1/3}n^{1/3}}+\frac{2 \alpha ^4-2 \alpha ^2 (108 t+1)+81 t^2}{432\kappa^2  t^{2/3}n^{2/3}}\no\\[8pt]
&-\frac{\alpha  \left[2(81 t-1) \alpha ^2+162 t^2-135 t+2\right]}{1944 t n }+O(n^{-4/3})\no,
\end{align}
where $\kappa= \sqrt[3]{2}$. It follows that the large $n$ asymptotic expansion for ${\ln D_n}(t)$ is
\begin{align}\label{dnt2}
{\ln D_n}(t) =& \:  {n^2}\ln n + \alpha  n\ln n+ \dfrac{12 \alpha ^2-5}{36}\ln n- \dfrac{3n^2}{2} - {c_3}(t,\al)n-\frac{9 t^{2/3}n^{2/3} }{4 \kappa}\no\\[8pt]
&+\frac{3 \alpha t^{1/3}n^{1/3}}{\kappa^2}   - {c_0}(t,\al )-\frac{\alpha  \left(2 \alpha ^2+27 t-2\right)}{36 \kappa t^{1/3}n^{1/3}}-\frac{2 \alpha ^4-2 \alpha ^2 (108 t+1)+81 t^2}{432\kappa^2  t^{2/3}n^{2/3}}\no\\[8pt]
&+\frac{\alpha  \left[2(81 t-1) \alpha ^2+162 t^2-135 t+2\right]}{1944 t n}+O(n^{-4/3}).
\end{align}

To know more information of the constants ${c_3}(t,\al)$ and ${c_0}(t,\al)$, taking a derivative of (\ref{dnt2}) with respect to $t$ and substituting it into (\ref{Pin}), we have
\begin{align}\label{ddd}
{H _n}(t) = &-t \dfrac{d}{dt}{c_3}(t,\alpha )  n -\frac{3  t^{2/3}n^{2/3}}{2 \kappa}+\frac{\alpha   t^{1/3}n^{1/3}}{\kappa^2} -t \dfrac{d}{d t} {c_0}(t,\alpha )+\frac{\alpha  \left(\alpha ^2-27 t-1\right)}{54 \kappa t^{1/3}n^{1/3}}\no\\
&+\frac{ \alpha ^4+\alpha ^2 (54 t-1)-81 t^2}{324\kappa^2 t^{2/3}n^{2/3}}+\frac{\alpha  \left(\alpha ^2+81 t^2-1\right)}{972 t n}+O(n^{-4/3}).
\end{align}
The combination of (\ref{dd}) and (\ref{ddd}) yields
\begin{align}
&t\frac{d}{dt}{c_3}(t,\alpha) = 0,\no\\[5pt]
&t\frac{d}{dt}{c_0}(t,\alpha  ) = \frac{6 \alpha ^2-18 t-1}{36}.\no
\end{align}
It follows that
\begin{subequations}\label{c3}
\begin{align}
&{c_3}(t,\alpha ) = \tilde{c}_3(\alpha ),\\[5pt]
&{c_0}(t,\alpha ) = \frac{\left(6 \alpha ^2-1\right) \ln t-18 t}{36}  +  \tilde{c}_0(\alpha ),
\end{align}
\end{subequations}
where $\tilde{c}_0(\alpha )$ and $\tilde{c}_3(\alpha )$  are constants depending only on $\alpha$.

By making use of (\ref{pn0}) and (\ref{dnt2}), together with (\ref{c3}), we have
\begin{align}\label{pn02}
\ln \left[(-1)^nP_n(0;t,\al)\right] =&\:{\ln D_n}(t,\al+1)-{\ln D_n}(t,\al)\no\\
=& \:  n\ln n + \frac{2\al+1}{3}\ln n+\left(\tilde{c}_3(\alpha )-\tilde{c}_3(\alpha+1 )\right)n +\frac{3 t^{1/3}n^{1/3}}{\kappa^2}-\frac{2\al+1}{6}\ln t\no\\[8pt]
&+  \tilde{c}_0(\alpha )-\tilde{c}_0(\alpha+1 )-\frac{2\al^2+2\alpha+9t  }{12 \kappa t^{1/3}n^{1/3}}-\frac{(2 \alpha+1)(\al^2+\al-54t)}{108\kappa^2  t^{2/3}n^{2/3}}\no\\[8pt]
&+\frac{2\al(\al+1)(81t-1)+9t(6t+1)}{648tn}+O(n^{-4/3}).
\end{align}
On the other hand, using (\ref{pnz}) we find after some elaborate computations
$$
\ln \left[(-1)^nP_n(0;t,\al)\right] \sim\left(n+\al+\frac{1}{2}\right)\ln\frac{X+Y}{2}-\left(\al+\frac{1}{2}\right)\ln Y-\frac{X-Y}{2}+\frac{t(X-Y)}{2Y^2},
$$
where $X=\frac{a+b}{2}$ and $Y=\sqrt{ab}$. Substituting (\ref{Y1}) and (\ref{b2}) into the above, we obtain as $n\rightarrow\infty$
\begin{align}\label{pn01}
\ln \left[(-1)^nP_n(0;t,\al)\right]
\sim& \:  n\ln n + \frac{2\al+1}{3}\ln n-n +\frac{3 t^{1/3}n^{1/3}}{\kappa^2}-\frac{2 \alpha +1}{6}  \ln (2 t)\no\\[8pt]
&-\frac{2\al^2+2\alpha+9t  }{12 \kappa t^{1/3}n^{1/3}}+O(n^{-2/3}).
\end{align}
Comparing (\ref{pn02}) with (\ref{pn01}) gives
\begin{align}
&\tilde{c}_3(\alpha+1 )-\tilde{c}_3(\alpha )=1,\no\\[5pt]
&\tilde{c}_0(\alpha+1 )-\tilde{c}_0(\alpha )=\frac{2\al+1}{6}\ln 2.\no
\end{align}
It follows that
\begin{subequations}\label{c3a}
\begin{align}
&\tilde{c}_3(\alpha )=\al+\hat{c}_3(\al),\\[5pt]
&\tilde{c}_0(\alpha )=\frac{\al^2\ln 2}{6}+\hat{c}_0(\al),
\end{align}
\end{subequations}
where $\hat{c}_3(\al)$ and $\hat{c}_0(\al)$ are both functions of $\al$ with a period of 1. Taking account of (\ref{dn1}), we have
\be\label{c3b}
\hat{c}_3(\al)=-\ln (2\pi).
\ee
The theorem is then established from the combination of (\ref{dnt2}), (\ref{c3}), (\ref{c3a}) and (\ref{c3b}).
\end{proof}
\begin{remark}
The form of the series expansion in (\ref{fna}) is consistent with that in (\ref{dn1}).
We are not able to determine the constant $\hat{c}_0(\al)$ completely for the singularly perturbed Laguerre weight problem using our method. The full asymptotic expansions of the Hankel determinants for regularly perturbed weights have been derived in previous work, see, e.g., \cite{Min2022,MF}. Note that our weight in (\ref{w1}) is not a special case of Charlier and Gharakhloo \cite[(1.2)]{CG} (a varying weight with the potential $n V(x)$) and can not be transformed to their case either. As a result, the large $n$ expansion form of our Hankel determinant in (\ref{dnt1}) is different from \cite[Theorem 1.2]{CG} since we have the fractional powers of $n$.
\end{remark}
\begin{theorem}
The normalized constant $h_n(t)$ has the large $n$ asymptotic expansion
\begin{align}
\ln h_n(t)=&\:2n\ln n - 2n + ( \alpha + 1 )\ln n +\ln (2\pi)  -\frac{3  t^{2/3}}{2 \kappa n^{1/3}}+\frac{\alpha   t^{1/3}}{\kappa^2 n^{2/3}}\no\\[8pt]
&+\frac{12 \alpha ^2+18 \alpha +7}{36 n}+\frac{(\alpha +1) \left(2 \alpha ^2-2 \alpha +27 t\right)}{108\kappa t^{1/3}n^{4/3} }+O(n^{-5/3}),\no
\end{align}
where $\kappa= \sqrt[3]{2}$.
\end{theorem}
\begin{proof}
From (\ref{d_nt1}) we have
\begin{align}\label{hn}
 \ln h_n(t)=\ln D_{n+1}(t)-\ln D_n(t).
\end{align}
Substituting (\ref{dnt1}) into (\ref{hn}) gives the desired result by taking a large $n$ limit.
\end{proof}

\section{Long-time asymptotics}
In this section, we fix $n\in\mathbb{N}$ and study the asymptotics of the recurrence coefficients $\alpha_n(t)$ and $\beta_n(t)$, the sub-leading coefficient $\mathrm{p}(n,t)$, the quantity $H_n(t) = t\frac{d}{{dt}}\ln {D_n}(t)$, the Hankel determinant $D_n(t)$ and the normalized constant $h_n(t)$ as $t  \to  + \infty$.
\begin{lemma}
The recurrence coefficients are expressed in terms of the quantity $H_n(t)$ as follows:
\begin{equation}\label{aa1}
\alpha_n(t)=2n+1+\alpha+H_n(t)-{H_{n+1}}(t),
\end{equation}
\begin{equation}\label{bb1}
 \beta_n(t)=n(n+\alpha)-{H_n}(t)+tH_n'(t).
\end{equation}
\end{lemma}
\begin{proof}
From (\ref{s11}) and (\ref{tR}), we have  (\ref{aa1}).
Using (\ref{ddt1}) and (\ref{r12}) we find
\be\label{dphi}
tH_n'(t)=r_n(t).
\ee
The combination of (\ref{r11}) and (\ref{dphi}) gives (\ref{bb1}).
\end{proof}
\begin{theorem}\label{th1t}
As $t  \to  + \infty$, the recurrence coefficients $\alpha_n(t)$ and $\beta_n(t)$ have the asymptotic expansions
\begin{subequations}\label{albe}
\begin{align}\label{d21}
\alpha_n(t)=&\:\sqrt{t}+\dfrac{2 \alpha +3 (2 n+1)}{4}+\dfrac{(2 \alpha +2 n+1) (2 \alpha +6 n+3)}{32 \sqrt{t}} \no\\[8pt]
 &-\dfrac{(2 \alpha +2 n+1) \left[\alpha  (4 n+2)+8 n^2+8 n+3\right]}{64 t}+O(t^{-{3/2}}),\\[8pt]
  \beta _n(t) =&\:\dfrac{n \sqrt{t}}{2}+\dfrac{2 \alpha  n +3 n^2}{4}+\dfrac{3 n (2 \alpha +2 n+1) (2 \alpha +2 n-1)}{64 \sqrt{t}}\no\\[8pt]
 &-\dfrac{n^2 (2 \alpha +2 n+1) (2 \alpha +2 n-1)}{32 t}+O(t^{-{3/2}}).\label{d22}
\end{align}
\end{subequations}
\end{theorem}
\begin{proof}
We use mathematical induction to prove this theorem. From (\ref{aa1}) and  (\ref{bb1}), and in view of the fact that $D_0(t)=1$, we have
$$
   \alpha_0(t)=1+\alpha-t \dfrac{d}{dt}\ln {\mu _0}(t),\qquad\qquad \beta_0(t)=0,
$$
where
$$
{\mu _0}(t)=2t^{\frac{\alpha +1 }{2} }K_{\alpha +1 }\big(2\sqrt t \big).
$$
It follows that
$$
\alpha_0(t)=\frac{\sqrt{t}\: K_{\alpha +2}\left(2 \sqrt{t}\right)}{K_{\alpha +1}\left(2 \sqrt{t}\right)}
$$
and it has the asymptotic expansion as $t  \to  + \infty$
$$
  {\alpha _{0}(t)} = \sqrt{t}+\frac{ 2 \alpha +3}{4}+\frac{(2 \alpha +1) (2 \alpha +3)}{32 \sqrt{t}}-\frac{(2 \alpha +1) (2 \alpha +3)}{64 t}+O(t^{-3/2}).
$$
According to the Toda-type system (\ref{tb}), we find
\begin{align}
  {\beta _{1}(t)} &=\al_0(t)-t\al_0'(t)\no\\
  &=\frac{\sqrt{t}}{2}+\frac{ 2\alpha +3}{4} +\frac{3 (2 \alpha +1) (2 \alpha +3)}{64 \sqrt{t}}-\frac{ (2 \alpha +1) (2 \alpha +3)}{32 t}+O(t^{-{3/2}}),\no\\
\al_1(t)&=\al_0(t)+2-t\frac{d}{dt}\ln \bt_1(t)\no\\
&=\sqrt{t}+\frac{2 \alpha +9}{4} +\frac{(2 \alpha +3) (2 \alpha +9)}{32\sqrt{t}} -\frac{(2 \alpha +3) (6 \alpha +19)}{64 t}+O(t^{-{3/2}}),\no
\end{align}
which is (\ref{albe}) with $n=1$.
Suppose that (\ref{albe}) is true, then using the Toda-type system (\ref{tb}) we obtain
\begin{align}
  \beta _{n+1}(t)=&\:\al_n(t) +\bt_n(t)-t\al_n'(t)\no\\[8pt]
  =&\:\dfrac{(n+1) \sqrt{t}}{2}+\dfrac{2 \alpha  (n+1)+3 (n+1)^2}{4}+\dfrac{3 (n+1)(2 \alpha +2 n+3) (2 \alpha +2 n+1)}{64 \sqrt{t}}\no\\[8pt]
 &-\dfrac{(n+1)^2 (2 \alpha +2 n+3) (2 \alpha +2 n+1)}{32 t}+O(t^{-3/2}),\no\\[8pt]
\alpha_{n+1}(t)=&\:\al_n(t)+2-t\frac{d}{dt}\ln\bt_{n+1}(t)\no\\[8pt]
=&\:\sqrt{t}+\dfrac{2 \alpha +3 (2 n+3)}{4}+\dfrac{(2 \alpha +2 n+3) (2 \alpha +6 n+9)}{32 \sqrt{t}}\no\\[8pt]
&-\dfrac{(2 \alpha +2 n+3) \left[\alpha  (4 n+6)+8 n^2+24 n+19\right]}{64 t}+O(t^{-{3/2}}),\no
\end{align}
which is (\ref{albe}) with $n\mapsto n+1$. Hence, the theorem is proved by induction.
\end{proof}
\begin{theorem}
As $t  \to  + \infty$, the sub-leading coefficient $\mathrm{p}(n,t)$ and the quantity $H_n(t) = t\frac{d}{{dt}}\ln {D_n}(t)$ have the asymptotic expansions
\be\label{pp1}
   \mathrm{p}(n,t) = -n \sqrt{t}-\frac{ n (2 \alpha +3 n)}{4}-\frac{n \left[4 (\alpha+n)^2-1\right]}{32\sqrt{t}}
   +O(t^{-1}),
   \ee
   \be\label{Pt}
       H_n(t)=-n \sqrt{t}+\frac{n (2 \alpha +n)}{4} -\frac{n \left[4 (\alpha+n)^2-1\right]}{32 \sqrt{t}} +O(t^{-1}).
   \ee
\end{theorem}
\begin{proof}
Substituting (\ref{albe}) into (\ref{pnt}) and letting $t  \to  + \infty$, we obtain (\ref{pp1}). The asymptotic expansion of $H_n(t)$ in (\ref{Pt}) comes from the relation (\ref{r12}).
\end{proof}
\begin{theorem}
As $t  \to  + \infty$, the Hankel determinant $D_n(t)$ has the asymptotic expansion
\be\label{dn3}
\ln D_n(t)=-2 n \sqrt{t}+\frac{n (2 \alpha +n)}{4}\ln t+\widetilde{C}(n)
 +\frac{n \left[4 (\alpha+n) ^2-1\right]}{16\sqrt{t}}+O(t^{-1}),
\ee
where $\widetilde{C}(n)$ is a constant, independent of $t$, given by
\be\label{con}
\widetilde{C}(n)=\dfrac{n\ln \pi}{2}-\dfrac{n(n-1)\ln 2}{2}+ \ln G(n+1)
\ee
and $G(\cdot)$ is the Barnes $G$-function \cite{ew}.
\end{theorem}
\begin{proof}
Since $H_n(t) = t\frac{d}{{dt}}\ln {D_n}(t)$,
from (\ref{Pt}) we have as $t\rightarrow+\infty$
\be\label{Dt2}
\ln D_n(t)=-2 n \sqrt{t}+\frac{n (2 \alpha +n)}{4}\ln t+\widetilde{C}(n)
 +\frac{n \left[4 (\alpha+n) ^2-1\right]}{16\sqrt{t}}+O(t^{-1}),
\ee
where $\widetilde{C}(n)$ is an integration constant, independent of $t$. From (\ref{bd}) it follows that
\begin{equation}\label{lnb}
\ln \beta_n(t)=\ln D_{n+1}(t)+\ln D_{n-1}(t)-2\ln D_{n}(t).
\end{equation}
By (\ref{d22}) we find as $t\rightarrow+\infty$
\be\label{beta3}
\ln \beta_n(t)=\frac{1}{2}\ln t + \ln \dfrac{n}{2} +\dfrac{2\alpha+3n}{2\sqrt{t}}+O(t^{-1}).
\ee
Substituting (\ref{Dt2}) and (\ref{beta3})  into (\ref{lnb}) and comparing the constant terms gives
\begin{equation}\label{CC1}
\widetilde{C}(n+1)+\widetilde{C}(n-1)-2\widetilde{C}(n)=\ln \dfrac{n}{2}.
\end{equation}
Note that
$$
D_0(t)=1, \qquad\qquad D_1(t)={\mu _0}=2t^{\frac{\alpha +1 }{2} }K_{\alpha +1 }\big(2\sqrt t \big)
$$
and as $t\rightarrow+\infty$
$$
{\mu _0}=\mathrm{e}^{-2\sqrt{t}}t^{\frac{\al}{2}+\frac{1}{4}}\sqrt{\pi}\left[1+\frac{(2\al+1)(2\al+3)}{16\sqrt{t}}+O(t^{-1})\right].
$$
It follows that
\begin{equation}\label{C10}
\widetilde{C}(0)=0,\qquad\qquad \widetilde{C}(1)=\dfrac{\ln \pi}{2}.
\end{equation}
The recurrence relation (\ref{CC1}) with the initial conditions (\ref{C10}) gives rise to the result in (\ref{con}).
This completes the proof.
\end{proof}
\begin{remark}
The long-time asymptotics (without higher-order terms) of the Hankel determinant in (\ref{dn3}) has also been recently derived by Chang, Eckhardt and Kostenko \cite[Lemma 6.7]{Chang} employing different method, which has been used to obtain the long-time asymptotics of the peakon solutions of the Camassa-Holm equation. In addition, the constant (\ref{con}) also appeared in the asymptotic expansion of the Hankel determinant for the generalized Airy weight as $t  \to  + \infty$ \cite{MF}.
\end{remark}
\begin{theorem}
As $t  \to  + \infty$, the normalized constant $h_n(t)$ has the asymptotic expansion
$$
\ln h_n(t)=-2 \sqrt{t}+\frac{ 2 \alpha +2 n+1}{4}\ln t+\widehat{C}(n)+\frac{ (2 \alpha +2 n+1) (2 \alpha +6 n+3)}{16\sqrt{t} } +O(t^{-1}),
$$
where the constant term $\widehat{C}(n)$ is
$$
\widehat{C}(n)=\frac{\ln \pi}{2}-n \ln2+\ln \Gamma(n+1),
$$
and $\Gamma(\cdot)$ is the Gamma function.
\end{theorem}
\begin{proof}
 Substituting (\ref{dn3}) into (\ref{hn}) and taking a large $n$ limit yields the result.
\end{proof}

\section{Conclusion}
In this paper, we give a comprehensive study of the orthogonal polynomials and Hankel determinants for a singularly perturbed Laguerre weight, including the differential and difference equations for the orthogonal polynomials and the recurrence coefficients, the zeros of the orthogonal polynomials and the large $n$ asymptotics and the long-time asymptotics for the problem. The orthogonal polynomials, Hankel determinants and associated matrix model for the singularly perturbed Laguerre weight have many important applications both in mathematics and physics, e.g., the distribution of the Wigner delay time, the statistics for zeros of the Riemann zeta function, the bosonic replica field theories, the peakon solutions of the Camassa-Holm equation. Finally, we point out that it is still a challenging task to completely determine the constant term in the large $n$ asymptotic expansion of the Hankel determinant in Theorem \ref{thm} and the problem remains open.
\section*{Acknowledgments}
We would like to thank Xiang-Ke Chang and Shuai-Xia Xu for helpful discussions. This work was partially supported by the National Natural Science Foundation of China under grant number 12001212, by the Fundamental Research Funds for the Central Universities under grant number ZQN-902 and by the Scientific Research Funds of Huaqiao University under grant number 17BS402.

\section*{Conflict of Interest}
The authors have no competing interests to declare that are relevant to the content of this article.
\section*{Data Availability Statement}
Data sharing is not applicable to this article as no datasets were generated or analysed during the current study.

\end{document}